\documentclass[11pt, eqno]{article}
\usepackage{bbm}
\usepackage{mathrsfs}
\usepackage{amsfonts}
\usepackage{amssymb}
\usepackage{graphicx}
\usepackage[all]{xy}
\usepackage{amsthm}
\usepackage{amsmath}
\usepackage{amsmath,amssymb,latexsym,color}
\usepackage[mathscr]{eucal}
\usepackage{CJK}
\usepackage{cases}
\usepackage{graphics}
\newenvironment{dedication}
        {\vspace{3ex}\begin{quotation}\begin{center}\begin{em}}
        {\par\end{em}\end{center}\end{quotation}}

\usepackage{psfrag}
\usepackage{subfigure}
\usepackage{color}

\usepackage{amssymb,latexsym}
\usepackage{amsmath,latexsym}
\usepackage{amscd}
\newcommand{\R}{{\mathbb R}}

\newtheorem{theorem}{Theorem}[section]

\newtheorem{thm}[theorem]{Theorem}

\newtheorem{remark}[theorem]{Remark}

\newtheorem{lemma}[theorem]{Lemma}

\newtheorem{proposition}[theorem]{Proposition}

\begin{document}

\title{Combinatorial formulas for some generalized Ekeland-Hofer-Zehnder capacities of convex polytopes}
\date{September 13, 2021}
\author{Kun Shi and Guangcun Lu
\thanks{Corresponding author
\endgraf \hspace{2mm} Partially supported
by the NNSF  11271044 of China.
\endgraf\hspace{2mm} 2010 {\it Mathematics Subject Classification.}
 53D05, 52B12 (primary), 53D05, 53C23 (secondary).}}
 %53D35, 53C23 (primary), 70H05, 37J05, 57R17 (secondary).}}
 \maketitle \vspace{-0.3in}
\begin{dedication}
\hfill{Dedicated to Professor Claude Viterbo on the occasion of his sixtieth birthday}
\end{dedication}
\vspace{0.1in}

%To Claude Viterbo on the Occasion of his Sixtieth Birthday

\abstract{
Motivated by Pazit Haim-Kislev's   combinatorial formula for the Ekeland-Hofer-Zehnder capacities of convex polytopes,
we give corresponding formulas for $\Psi$-Ekeland-Hofer-Zehnder and coisotropic Ekeland-Hofer-Zehnder capacities of convex
polytopes introduced by the second named author and others recently.
Contrary to Pazit Haim-Kislev's subadditivity result for the Ekeland-Hofer-Zehnder capacities of convex domains,
we  show that the coisotropic Hofer-Zehnder capacities
 satisfy the  superadditivity for suitable hyperplane cuts of two-dimensional convex domains.
} \vspace{-0.1in}
\medskip\vspace{12mm}

\maketitle \vspace{-0.5in}

%%%%%%%%%%%%%%%%%%%%%%%%%%%%%%%%%%%%%%%%%%%%%%%%%%%%%%%%%%%%%%%%%%%%%%%%%%%%%%%%%%%%%%%%%%%%%%%%%%%%%%%%%%%%%%%%%%%%%%%%%%
%%\noindent{\it Keywords}: Ekeland-Hofer symplectic capacity; Hofer-Zehnder symplectic capacity; Brunn-Minkowski type inequality; Minkowski billiard \vspace{2mm}
%%%%%%%%%%%%%%%%%%%%%%%%%%%%%%%%%%%%%%%%%%%%%%%%%%%%%%%%%%%%%%%%%%%%%%%%%%%%%%%%%%%%%%%%%%%%%%%%%%%%%%%%%%%%%%%%%%%%%%%%%%
%Critical groups;

\tableofcontents

\section{Introduction and results}
\setcounter{equation}{0}

Symplectic capacities are important invariants in studies of symplectic topology.
  Different symplectic capacities  measure the ``symplectic size" of sets from different views.
  Precise computations of them are usually difficult.

For  a compact convex  domain  $K$ with smooth boundary $\mathcal{S}=\partial K$
in the standard symplectic Euclidean space $(\mathbb{R}^{2n},\omega_0)$,
  Ekeland-Hofer \cite{EH89} (see also  \cite{Sik90}) and  Hofer-Zehnder \cite{HoZe90} showed, respectively, that
  its Ekeland-Hofer capacity $c_{\rm EH}(K)$ and Hofer-Zehnder capacity $c_{\rm HZ}(K)$ were equal to
    \begin{eqnarray}\label{e:EHZ}
c_{\rm EHZ}(K):=\min\{A(x)>0\,|\,x\;\text{is a closed characteristic on}\;\mathcal{S}\}
\end{eqnarray}
 (called the Ekeland-Hofer-Zehnder capacity below), where by a  \textsf{closed characteristic} on $\mathcal{S}$ we mean
  a $C^1$ embedding $z$ from $S^1=[0, T]/\{0,T\}$ into $\mathcal{S}$ satisfying
   $\dot{z}(t)\in (\mathcal{L}_{\mathcal{S}})_{z(t)}$ for all $t\in [0, T]$, where
$$
\mathcal{L}_{\mathcal{S}}=\{(x,\xi)\in T\mathcal{S}\mid \omega_{0x}(\xi,\eta)=0 ,\forall \eta\in T_x\mathcal{S}\}
$$
and the action of a path $z\in W^{1,2}([0,T], \mathbb{R}^{2n})$ is defined by
  \begin{equation}\label{e:action1}
A(z)=\frac{1}{2}\int_0^T\langle -J\dot{z},z\rangle dt
\end{equation}
  with $J=\left(
           \begin{array}{cc}
             0 & -I_n \\
             I_n & 0 \\
           \end{array}
         \right)$,  where $z\in W^{1,2}([0,T],\mathbb{R}^{2n})$ if $z$ is absolutely continuous and
         $$\int_{0}^T\|z(t)\|^2dt<\infty \quad\text{and} \quad \int_{0}^T\|\dot{z}(t)\|^2dt<\infty.$$
  We equip $H^1([0,T],\mathbb{R}^{2n}):=W^{1,2}([0,T], \mathbb{R}^{2n})$ the natural Sobolev norm:
  $$\|z\|_{W^{1,2}}:=\left(\int_0^T\|z(t)\|^2+\|\dot{z}(t)\|^2dt\right)^{\frac{1}{2}}.$$
When the smoothness assumption of the boundary $\mathcal{S}$ is thrown away, then (\ref{e:EHZ}) is still true if ``closed characteristic'' in the right side of (\ref{e:EHZ}) may be replaced by ``generalized closed characteristic", where a \textsf{generalized closed characteristic} on $\mathcal{S}$
is a $T$-periodic nonconstant absolutely continuous curve $z:\mathbb{R}\rightarrow\mathbb{R}^{2n}$ (for some $T>0$) such that
 $z(\mathbb{R})\subset\mathcal{S}$ and $\dot{z}(t)\in JN_{\mathcal{S}}(z(t))$ a.e., where $N_{\mathcal{S}}(x)=\{y\in\mathbb{R}^{2n}\mid\langle u-x,y\rangle\leqslant 0, \forall u\in K\}$ is the normal cone to $K$ at $x\in\mathcal{S}$.
The action of such a generalized closed characteristic $x:[0,T]\rightarrow\mathcal{S}$ is still defined by (\ref{e:action1}).

  In general,  it is  difficult to compute $c_{\rm EHZ}(K)$ by finding minimal closed characteristics with (\ref{e:EHZ}).
   If $K$ is a convex polytope with   $(2n-1)$-dimensional facets $\{F_i\}_{i=1}^{{\bf F}_K}$,
   $n_i$ is the unit outer normal to $F_i$, and  $h_i = h_K(n_i)$ the ``\textbf{oriented height}" of $F_i$ given by
  the support function of $K$, $h_K(y) := \sup_{x \in K} \langle x,y \rangle$,
   starting from (\ref{e:EHZ})   Pazit Haim-Kislev \cite{PH19} recently established the following beautiful
 combinatorial formula for $c_{\rm EHZ}(K)$:
  \begin{equation}\label{e:Comb formula}
c_{\rm EHZ}(K) = \frac{1}{2} \left[ \max_{\sigma\in S_{{\bf F}_K}, (\beta_i) \in M(K)}  \sum_{1 \leq j < i \leq {\bf F}_K}{} \beta_{\sigma(i)} \beta_{\sigma(j)} \omega_0(n_{\sigma(i)},n_{\sigma(j)}) \right]^{-1},
\end{equation}
where $S_{{\bf F}_K}$ is the symmetric group on ${\bf F}_K$ letters and
$$
M(K) = \left\{ (\beta_i)_{i=1}^{{\bf F}_K} \,\bigg|\, \beta_i \geq 0, \sum_{i=1}^{{\bf F}_K} \beta_i h_i = 1, \sum_{i=1}^{{\bf F}_K} \beta_i n_i = 0 \right\}.
$$
   As an important application,  Pazit Haim-Kislev \cite{PH19} proved a subadditivity property of the capacity $c_{\rm EHZ}$
   for hyperplane cuts of arbitrary convex domains, which solved a special case of the subadditivity conjecture for capacities
 (\cite{AKP14}).

Recently, motivated by Clarke \cite{Cl82, Cl83} and  Ekeland \cite{Ek17}
Rongrong Jin and the second named author introduced relative versions (or generalizations)  of
the Ekeland-Hofer capacity and the Hofer-Zehnder capacity   in \cite{JinLu1916}.
 Precisely, for  a symplectic manifold $(M,\omega)$ and for a $\Psi\in{\rm Symp}(M,\omega)$ with ${\rm Fix}(\Psi)\ne\emptyset$,
  we defined a relative version of the Hofer-Zehnder capacity $c_{\rm HZ}(M,\omega)$ of $(M, \omega)$ with respect to $\Psi$,
  $c^\Psi_{\rm HZ}(M,\omega)$, which becomes $c_{\rm HZ}(M,\omega)$ if $\Psi=id_M$.
 For a symplectic  matrix  $\Psi\in{\rm Sp}(2n,\mathbb{R})$ with ${\rm Fix}(\Psi)\ne\emptyset$,
 and for each $B\subset\mathbb{R}^{2n}$ such that $B\cap {\rm Fix}(\Psi)\neq \emptyset$,
   we also introduced  a relative version of the Ekeland-Hofer capacity $c_{\rm EH}(B)$ of $B$ with respect to $\Psi$,
  $c^\Psi_{\rm EH}(B)$, which becomes $c_{\rm EH}(B)$ if $\Psi=I_{2n}$.
  If a compact convex  domain  $K\subset\mathbb{R}^{2n}$ with  boundary $\mathcal{S}=\partial K$
contains a fixed point of $\Psi\in{\rm Sp}(2n,\mathbb{R})$ in the interior of it, we proved in \cite{JinLu1916}:
  \begin{eqnarray}\label{e:PsiEHZ}
c^\Psi_{\rm EH}(K)=c^\Psi_{\rm HZ}(K)=\min\{A(x)>0\,|\,x\;\text{is a generalized $\Psi$-characteristic on}\;\mathcal{S}\},
\end{eqnarray}
where a \textsf{generalized $\Psi$-characteristic} on $\mathcal{S}$ is
a nonconstant absolutely continuous curve $z:[0,T]\rightarrow\mathbb{R}^{2n}$ (for some $T>0$) such that
 $z([0,T])\subset\mathcal{S}$, $z(T)=\Psi z(0)$ and $\dot{z}(t)\in JN_{\mathcal{S}}(z(t))$ a.e., where $N_{\mathcal{S}}(x)$
  is the normal cone to $K$ at $x\in\mathcal{S}$ as above, and the action $A(z)$ of $z$ is still defined by (\ref{e:action1}).
(If $\mathcal{S}$ is  $C^{1,1}$-smooth, ``generalized closed characteristic'' in the right side of (\ref{e:PsiEHZ})
may be replaced by ``closed characteristic", where a \textsf{$\Psi$-characteristic} on $\mathcal{S}$ is a $C^1$ embedding $z$ from $[0, T]$ (for some $T >0$) into $\mathcal{S}$ such that  $z(T)=\Psi z(0)$ and $\dot{z}\in (\mathcal{L}_{\mathcal{S}})_{z(t)}$
for all $t\in [0,T]$).
Our first result  is an analogue of (\ref{e:Comb formula}) for $c^\Psi_{\rm EHZ}(K):=c^\Psi_{\rm EH}(K)=c^\Psi_{\rm HZ}(K)$.

\begin{theorem}\label{th:main}
Let $K$ be a convex polytope as above (\ref{e:Comb formula}).
Suppose that $\Psi\in{\rm Sp}(2n,\mathbb{R})$  has a fixed point sitting in the interior of $K$.
Then
$$
c^{\Psi}_{{\rm EHZ}}(K)=\min_{\bigl((\beta_i)_{i=1}^{{\bf F}_K},v, \sigma\bigr)\in M_{\Psi}(K)}\frac{2}{4\sum_{1\leqslant j<i\leqslant{\bf F}_K}
\beta_{\sigma(i)}\beta_{\sigma(j)}\omega_0(n_{\sigma(j)},n_{\sigma(i)})-\omega_0(\Psi v, v)},
$$
where
$$
M_{\Psi}(K)=\left\{\big((\beta_i)_{i=1}^{{\bf F}_K},v, \sigma\bigr)\,\bigg|\,\begin{array}{ll}
&\sigma\in
S_{{\bf F}_K},\;\beta_i\geqslant
0,\;\sum_{i=1}^{{\bf F}_K}\beta_ih_i=1,\;\sum_{i=1}^{{\bf F}_K}2\beta_i Jn_i=\Psi v-v, \\
&4\sum_{1\leqslant j<i\leqslant{\bf F}_K}
\beta_{\sigma(i)}\beta_{\sigma(j)}\omega_0(n_{\sigma(j)},n_{\sigma(i)})>\omega_0(\Psi v, v),\;v\in
E_{\Psi}
\end{array}
\right\}
$$
with  $E_{\Psi}$ being the orthogonal complement of ${\rm Ker}(\Psi-I_{2n})$ in $\mathbb{R}^{2n}$.
 \end{theorem}

{\it Note}: Under our convention $\langle x,y\rangle=\omega_0(x, Jy)$, %is different from \cite{PH19}, and so
$\omega_0(n_{\sigma(i)},n_{\sigma(j)})$ in (\ref{e:Comb formula}) should be changed into
$\omega_0(n_{\sigma(j)},n_{\sigma(i)})$.

 Lisi  and Rieser \cite{LR13} introduced the notion of a coisotropic capacity and constructed
a coisotropic Hofer-Zehnder capacity, which is a relative version of the  Hofer-Zehnder capacity
with respect to a coisotropic submanifold.
 Rongrong Jin and the second named author recently constructed a relative version of the Ekeland-Hofer capacity
with respect to a special class of coisotropic subspaces in \cite{JinLu1918}.
 Consider coisotropic subspaces of $(\mathbb{R}^{2n}, \omega_0)$,
$$
\mathbb{R}^{n,k}=\{x\in\mathbb{R}^{2n}|x=(q_1,\cdots,q_n,p_1,\cdots,p_k,0,\cdots,0)\},\quad k=0,\cdots,n.
$$
The isotropic leaf through $x\in \mathbb{R}^{n,k}$ is $x+V_0^{n,k}$, where
$$
V_0^{n,k}=\{x\in\mathbb{R}^{2n}\,|\,x=(0,\cdots,0,q_{k+1},\cdots,q_n,0,\cdots,0)\}.
$$
The leaf relation $\sim$ on $\mathbb{R}^{n,k}$ is that $x\sim y$ if and only if $y\in x+V_0^{n,k}$.
From now on we fix an integer $0\leqslant k<n$ and assume that $K\subset\mathbb{R}^{2n}$ is a compact convex  domain   with $C^{1,1}$-smooth boundary $\mathcal{S}=\partial K$
and satisfying ${\rm Int}(K)\cap \mathbb{R}^{n,k}\ne\emptyset$.
A nonconstant absolutely continuous curve $z:[0,T]\rightarrow\mathbb{R}^{2n}$ (for some $T>0$)
 is called  a \textsf{generalized leafwise chord} (abbreviated GLC) on $\mathcal{S}$ for $\mathbb{R}^{n,k}$
if  $z([0,T])\subset\mathcal{S}$, $\dot{z}(t)\in JN_{\mathcal{S}}(z(t))$ a.e., $z(0),z(T)\in\mathbb{R}^{n,k}$ and $z(0)-z(T)\in V_0^{n,k}$.
 The action $A(z)$ of such a chord is still defined by (\ref{e:action1}).
 In \cite{JinLu1917, JinLu1918} Rongrong Jin and the second named author
 proved respectively that the coisotropic Hofer-Zehnder capacity $c_{\rm LR}(K, K\cap\mathbb{R}^{n,k})$ of $K$ relative to $\mathbb{R}^{n,k}$
 and the coisotropic Ekeland-Hofer capacity $c^{n,k}(K)$ of $K$ relative to $\mathbb{R}^{n,k}$ satisfy
  \begin{eqnarray}\label{e:coiEHZ}
c_{\rm LR}(K, K\cap\mathbb{R}^{n,k})=c^{n,k}(K)=
\min\{A(x)>0\;|\;x\;\hbox{is a GLC on\;$\mathcal{S}$\;for \;$\mathbb{R}^{n,k}$}\}.
\end{eqnarray}
 Here is our second result.

\begin{theorem}\label{th:main2}
Let $K$ be a convex polytope as above (\ref{e:Comb formula}). Suppose $K\cap\mathbb{R}^{n,k}\neq\emptyset$.
 Then
 $$
c_{{\rm LR}}(K, K\cap\mathbb{R}^{n,k})=\frac{1}{2}\min_{((\beta_i)_{i=1}^{{\bf F}_K},\sigma)\in M(K)}\frac{1}{\sum_{1\leqslant
j<i\leqslant{\bf F}_K}
\beta_{\sigma(i)}\beta_{\sigma(j)}\omega_0(n_{\sigma(j)},n_{\sigma(i)})},
$$
 where
\begin{equation}\label{e:1.6}
M(K)=\left\{((\beta_i)_{i=1}^{{\bf F}_K},\sigma)\,\bigg|\,\begin{array}{ll}
&\beta_i\geqslant
0,\;\sum_{i=1}^{{\bf F}_K}\beta_ih_i=1,\;\sum_{i=1}^{{\bf F}_K}\beta_iJ n_i\in V_0^{n,k},\\
&\sum_{1\leqslant
j<i\leqslant{\bf F}_K}
\beta_{\sigma(i)}\beta_{\sigma(j)}\omega_0(n_{\sigma(j)},n_{\sigma(i)})>0,\;\sigma\in S_{{\bf F}_K}
\end{array}
\right\}.
\end{equation}
\end{theorem}

Unlike Ekeland-Hofer-Zehnder capacity, one cannot expect that the coisotropic Hofer-Zehnder
capacity satisfies the  subadditivity as stated in \cite[Theorem~1.8]{PH19} in general.
In fact, when $n=1$ and $k=0$, our following result is  opposite to the expected one.

\begin{theorem}\label{th:converse}
Let $D\subset\mathbb{R}^2$ be a convex domain satisfying $D\cap\mathbb{R}^{1,0}\neq\emptyset$,
and let $L\subset\mathbb{R}^{2}$ be a straight line through $D$
such that $L\neq\mathbb{R}^{1,0}$ and $D\cap L\cap\mathbb{R}^{1,0}\neq\emptyset$.
Denote by $D_1$ and $D_2$ the two parts divided by $L$. Then
\begin{equation}\label{e:1.7}
c_{\rm LR}(D, D\cap\mathbb{R}^{1,0})\geq c_{\rm LR}(D_1, D_1\cap\mathbb{R}^{1,0})+c_{\rm LR}(D_2, D_2\cap\mathbb{R}^{1,0}).
\end{equation}
\end{theorem}

\begin{remark}\label{rm:1}
{\rm Inequality (\ref{e:1.7}) is sharp, and it can be strict in some cases.
Consider the following example.
Let  $P=\{(x,y)\,|\,|x|\leqslant1, |y|\leqslant 1\}$ and $L=\{(x,x)\,|\, x\in\mathbb{R}$\}. Then $L$ divides $P$ into two parts  $P_1:=\{(x,y)\,|\, x\leqslant y\}\cap P$ and  $P_2:=\{(x,y)\,|\, x\geqslant y\}\cap P$.
Using Theorem~\ref{th:main2}, we can easily compute $c_{\rm LR}(P, P\cap\mathbb{R}^{1,0})=2$, $c_{\rm LR}(P_1, P_1\cap\mathbb{R}^{1,0})=c_{\rm LR}(P_2, P_2\cap\mathbb{R}^{1,0})=\frac{1}{2}$.
Thus
$$
c_{\rm LR}(P, P\cap\mathbb{R}^{1,0})>c_{\rm LR}(P_1, P_1\cap\mathbb{R}^{1,0})+c_{\rm LR}(P_2, P_2\cap\mathbb{R}^{1,0}).
$$
Moreover, for any $t\in (-1,1)$, the line $L_t:=\{(t,y)\,|\,y\in\mathbb{R}\}$
divides $P$ into two parts
$$
P_+:=\{(x,y)\in P\,|\, x\ge t\}\quad\hbox{and}\quad P_-:=\{(x,y)\in P\,|\, x\le t\}.
$$
It is easily computed that
$$
c_{\rm LR}(P_+, P_+\cap\mathbb{R}^{1,0})=1-t\quad\hbox{and}\quad c_{\rm LR}(P_-, P_-\cap\mathbb{R}^{1,0})=1+t,
$$
and hence $c_{\rm LR}(P_+, P_+\cap\mathbb{R}^{1,0})+ c_{\rm LR}(P_-, P_-\cap\mathbb{R}^{1,0})=c_{\rm LR}(P, P\cap\mathbb{R}^{1,0})$.

In higher dimensions, we have $c_{\rm LR}(G, G\cap\mathbb{R}^{n,n})=c_{\rm EHZ}(G)$ for any nonempty convex domain $G\subset\R^{2n}$.
%the coisotropic Hofer-Zehnder
%capacity  coincides with the Ekeland-Hofer-Zehnder capacity when picking $\mathbb{R}^{n,k}$ with $k=n$.
Thus some coisotropic Hofer-Zehnder capacities of higher dimensions have subadditivity
because of the subadditivity of $c_{\rm EHZ}$ under the conditions of \cite[Theorem~1.8]{PH19}.
There is no nice result in more general case yet.
}
\end{remark}

For the symmetrical Hofer-Zehnder symplectic capacity of a symmetric convex domain in $\R^{2n}$ introduced by Liu and Wang \cite{LW12},
using a representation formula of it given by Rongrong Jin and the second named author in \cite{JinLu20}
one is able to generalize the formula in \cite{PH19}, but this is outside the scope of this paper
and would appear elsewhere.

%Using the representation formula in \cite{JinLu20}, one is able
%to generalize the formula in \cite{PH19} for an additional generalization of the
%Hofer-Zehnder capacity: the symmetrical symplectic capacity introduced
%by Liu and Wang in \cite{LW12}, but this is outside the scope of this paper
%and would appear elsewhere.

This paper is organized as follows.
In the next section  we collect detailed conclusions coming from
 \cite[\S4.1]{JinLu1916} and \cite[\S3.1]{JinLu1917} about proofs of
representation formulas of the $\Psi$-Ekeland-Hofer-Zehnder capacity and
the coisotropic Ekeland-Hofer-Zehnder capacity for convex bodies in $\mathbb{R}^{2n}$, respectively.
Then we generalize some results on piecewise affine loops in \cite[\S 3]{PH19} to piecewise affine paths
in Section~\ref{sec:affine paths}.  Theorem~\ref{th:main} will be proved in Section~\ref{sec:main}.
Finally, in Section~\ref{sec:main-con} we prove
 Theorems~\ref{th:main2},~\ref{th:converse}.

\section{Preliminaries}\label{sec:pre}

For simplicity of the reader's convenience we list two results, which come from
\cite[Section~4.1]{JinLu1916} and \cite[Section~3.1]{JinLu1917}, respectively.

Let $K \subset \R^{2n}$ be a compact convex  domain  $K$ with  boundary $\mathcal{S}=\partial K$
and with $0\in{\rm Int}(K)$. Denote by  $H_K=(j_K)^2$ the square of the Minkowski functional $j_K$ of $K$,
 and by $H_K^*$ the Legendre transformation of $H_K$ defined by
$$
H_K^*(w)=\max_{\xi\in\mathbb{R}^{2n}}(\langle x,\xi\rangle-H_K(\xi)).
$$
 Then $h_K^2=4H_K^\ast$ (see e.g.\cite{AO14}).

Given  $\Psi\in{\rm Sp}(2n,\mathbb{R})$ let $E_{\Psi}$ be the orthogonal complement of ${\rm Ker}(\Psi-I_{2n})\subset\mathbb{R}^{2n}$
with respect to the standard  inner product in $\mathbb{R}^{2n}$. (In \cite{JinLu1916} we wrote ${\rm Ker}(\Psi-I_{2n})$ and $E_{\Psi}$ as $E_1$ and $E_1^{\bot}$,
respectively.) Define
 $$
 \mathcal{F}_\Psi=\{x\in W^{1,2}([0,1],\mathbb{R}^{2n}) \,|\,x(1)=\Psi x(0) \text{ and } x(0)\in E_\Psi\},
 $$
which was denoted by $\mathcal{F}$ in \cite{JinLu1916}.  If ${\rm dim}E_\Psi=0$, the problem reduces to the periodic case. So we only consider the non-periodic case in which ${\rm dim} E_\Psi\geqslant 1$. Define
$$
\mathcal{A}_\Psi=\{x\in\mathcal{F}_\Psi\,|\, A(x)=1\},
 $$
where $A(x)$ is defined by (\ref{e:action1}) with $T=1$, and
$$
I_K:\mathcal{F}_\Psi\rightarrow\mathbb{R},\; x\mapsto\int_0^1 H_K^*(-J\dot{x}).
$$
By Theorems~1.8, 1.9, Remark~1.10 and arguments in \cite[\S4.1]{JinLu1916} we have

\begin{thm}\label{th:convex}
Under the above assumptions, $I_K$ attains its minimum $\min_{x\in\mathcal{A}_\Psi}I_K(x)$ over $\mathcal{A}_\Psi$,
which is positive. For each minimier $u$ of $I_K$ over $\mathcal{A}_\Psi$,
there exists $a_0\in {\rm Ker}(\Psi-I_{2n})$  such that the $W^{1,2}$-path
\begin{equation}\label{e:repara1}
[0, I_K(u)]\ni t\mapsto x^*(t)=\sqrt{I_K(u)}u(t/I_K(u))+a_0/\sqrt{I_K(u)}
\end{equation}
satifies $A(x^{\ast})=I_K(u)=c^\Psi_{\rm EHZ}(K)$ and
\begin{equation}\label{e:repara2}
\left\{
   \begin{array}{l}
     -J\dot{x}^\ast(t)\in\partial {H}_K(x^\ast(t)), \;{\rm a.e.},\\
     x^\ast(T)=\Psi x^\ast(0)\quad\hbox{and}\quad x^\ast([0,T])\subset\partial K;
   \end{array}
   \right.
\end{equation}
in particular $x^{\ast}$ is a generalized $\Psi$-characteristic  on $\partial K$
because
\begin{equation}\label{e:subdifferential}
\partial {H}_K(x)=\{v\in N_{\partial K}(x)\,|\, \langle x,v\rangle=2\}\quad\forall x\in\partial K.
\end{equation}
(cf. Lemma~2 of \cite[Chap.V, \S1]{Ek90}).
Conversely, if $z:[0,T]\to \partial K$ is a generalized $\Psi$-characteristic  on $\partial K$ with action $A(z)=c^\Psi_{\rm EHZ}(K)$, then
(by \cite[Lemma~4.2]{JinLu1916}) there is a differentiable homeomorphism
$\varphi:[0, T]\to [0, T]$ with an absolutely continuous inverse
$\psi:[0, T]\to [0, T]$ such that $z^\ast=z\circ\varphi$
is a $W^{1,\infty}$-map with action $A(z^\ast)=A(z)=T$ and satisfying (\ref{e:repara2});
moreover we can choose $b\in{\rm Ker}(\Psi-I_{2n})$ so that the path $u:[0,1]\to\mathbb{R}^{2n}$ defined by $u(t)=z^\ast(Tt)/\sqrt{T}+ b$
belongs to  $\mathcal{A}_\Psi$ and satisfies $I_K(u)=T$, i.e., $u$ is a minimier $u$ of $I_K$ over $\mathcal{A}_\Psi$.
When this $K$ is also a convex polytope as above (\ref{e:Comb formula}),  then there holds
\begin{equation}\label{e:inclusion}
\dot{u}(t)=\sqrt{T}\dot{z}^\ast(Tt)\in \sqrt{T}{\rm conv}\{p_i\,|\, \sqrt{T}(u(t)-b)\in F_i\}, \;{\rm a.e.}
\end{equation}
where $p_i=\frac{2}{h_i}Jn_i$.
\end{thm}

  In order to see the final claim, note that for each $i=1,\cdots,{\bf F}_{K}$, $H_K$ is smooth at each relative interior point $x$ of $F_i$
and the subdifferential $\partial H_K(x)=\{\nabla H_K(x)\}=\{\frac{2}{h_i}n_i\}$.
For any $x\in\partial K$ we have $\partial H_K(x)={\rm conv}\{\frac{2}{h_i}n_i\,|\, x\in F_i\}$
(cf. \cite[page 445]{PH19}),  and therefore $J\partial H_K(x)={\rm conv}\{p_i\,|\, x\in F_i\}$.
(The outward normal cone of $K$ at $x\in\partial K$,  $N_{\partial K}(x)$,
  is equal to $\mathbb{R}_+{\rm conv}\{n_i:x\in F_i\}$.)

Fix an integer $0\le k<n$. Following \cite{JinLu1917} consider  the Hilbert subspace of $W^{1,2}([0,1],\mathbb{R}^{2n})$,
$$
\mathscr{F}_2:=\left\{x\in W^{1,2}([0,1],\mathbb{R}^{2n})\,\Big|\,x(0), x(1)\in\mathbb{R}^{n,k},\;x(1)\sim x(0),\;\int_0^1x(t)dt\in JV_0^{n,k}\right\}
$$
(where $x(1)\sim x(0)$ means $x(1)-x(0)\in V^{n,k}_0$), its subset
$\mathcal{A}_2 =\{x\in \mathscr{F}_2\,|\,A(x)=1 \}$,
%\begin{eqnarray}\label{e:Brunn.1}
%\mathcal{A}_2 =\{x\in \mathscr{F}_2\,|\,A(x)=1 \},
%\end{eqnarray}
and the related convex functional
$$
I_2:\mathscr{F}_2\to\mathbb{R},\;x\mapsto \int_0^1 H_K^{\ast}(-J\dot{x}(t))dt.
$$

From \cite[\S3.1]{JinLu1917}, we obtain the following corresponding result of Theorem~\ref{th:convex}.

\begin{thm}\label{th:cosi}
Under the above assumptions, $I_2$ attains its minimum $\min_{x\in\mathcal{A}_2}I_2(x)$ over $\mathcal{A}_2$,
which is positive. For each minimier $u$ of $I_2$ over $\mathcal{A}_2$, there exists
${\bf a}_0\in \mathbb{R}^{n,k}$ such that the $W^{1,2}$-path
\begin{eqnarray}\label{e:action3}
[0, 1]\ni t\mapsto x^*(t):=\sqrt{I_2(u)}u(t)+ {\bf a}_0/\sqrt{I_2(u)}
\end{eqnarray}
satisfies $A(x^{\ast})=I_2(u)=c_{{\rm LR}}(K, K\cap\mathbb{R}^{n,k})=c^{n,k}(K)$ and
\begin{equation}\label{e:repara3}
\left\{
   \begin{array}{l}
     -J\dot{x}^\ast(t)=\partial {H}_K(x^\ast(t)), \;{\rm a.e.},\;x^\ast(0), x^\ast(1)\in \mathbb{R}^{n,k},\\
     x^\ast(1)-x^\ast(0)\in V_0^{n,k}\quad\hbox{and}\quad x^\ast([0,1])\subset\partial K;
   \end{array}
   \right.
\end{equation}
in particular $x^\ast$ is a generalized leafwise chord   on $\partial K$  for $\mathbb{R}^{n,k}$
because of (\ref{e:subdifferential}).
Conversely, if $z:[0,T]\to \partial K$ is a generalized leafwise chord  on $\partial K$ with action $A(z)=c^{n,k}(K)$  for $\mathbb{R}^{n,k}$, then
(by \cite[Lemma~4.2]{JinLu1916}) there is a differentiable homeomorphism
$\varphi:[0, T]\to [0, T]$ with an absolutely continuous inverse
$\psi:[0, T]\to [0, T]$ such that $z^\ast=z\circ\varphi$
is a $W^{1,\infty}$-map with action $A(z^\ast)=A(z)=T$ and satisfying
\begin{equation}\label{e:repara4}
\left\{
   \begin{array}{l}
     -J\dot{z}^\ast(t)=\partial {H}_K(z^\ast(t)), \;{\rm a.e.},\;z^\ast(0), z^\ast(T)\in \mathbb{R}^{n,k},\\
     z^\ast(T)-z^\ast(0)\in V_0^{n,k}\quad\hbox{and}\quad z^\ast([0,T])\subset\partial K;
   \end{array}
   \right.
\end{equation}
moreover the path $u:[0,1]\to\mathbb{R}^{2n}$ defined by
\begin{equation}\label{e:repara5}
u(t)=\frac{1}{\sqrt{T}}z^\ast(Tt)- \frac{1}{\sqrt{T}}P_{n,k}\int^1_0 z^\ast(Tt)dt
\end{equation}
where  $P_{n,k}:\mathbb{R}^{2n}=JV^{n,k}_0\oplus \mathbb{R}^{n,k}\to\mathbb{R}^{n,k}$
is the orthogonal projection, belongs to  $\mathcal{A}_2$ and satisfies $I_2(u)=T$, i.e., $u$ is a
minimier $u$ of $I_2$ over $\mathcal{A}_2$.
When this $K$ is also a convex polytope as above (\ref{e:Comb formula}), there holds
$$
\dot{u}(t)=\sqrt{T}\dot{z}^\ast(Tt)\in \sqrt{T}{\rm conv}\{p_i\,|\, \sqrt{T}(u(t)-b)\in F_i\}, \;{\rm a.e.}
$$
where $p_i=\frac{2}{h_i}Jn_i$ and $b=- \frac{1}{\sqrt{T}}P_{n,k}\int^1_0 z^\ast(Tt)dt$.
\end{thm}

The final claim is obtained as below Theorem~\ref{th:convex}.

\section{Piecewise affine paths}\label{sec:affine paths}

In this section we will generalize some results on piecewise affine loops in \cite[\S 3]{PH19} to piecewise affine paths.

Recall in \cite[Definition~3.2]{PH19} that a finite sequence of disjoint open intervals $(I_i)_{i=1}^m$ is called a \textsf{partition} of $[0,1]$
if there exists an increasing sequence of numbers $0 = \tau_0 \leq \tau_1 \leq \ldots \leq \tau_m = 1$ with $I_i = (\tau_{i-1},\tau_i)$.
(Note that the open interval $I_i$ may be empty!)
As usual let $\chi_I$ denote the characteristic function of a subset $I\subset\mathbb{R}$.
   A path $z\in H^1([0,1],\mathbb{R}^{2n})$ is said to be \textsf{piecewise affine} if $\dot{z}$ can be written as $\dot{z}(t) = \sum_{j=1}^m \chi_{I_j}(t) w_j$ for almost every $t \in [0,1]$, where $(I_j)_{j=1}^m$ is a partition of $[0,1]$ and $(w_j)_{j=1}^m\in\mathbb{R}^{2n}$ is a finite sequence of vectors.

\begin{lemma}[\hbox{\cite[Lemma~3.1]{PH19}}]\label{lem:PH3.1}
Fix a set of vectors $v_1,\cdots,v_k\in\mathbb{R}^{2n}$. Suppose $z\in H^1([0,1],\mathbb{R}^{2n})$ satisfies that for almost every $t\in[0,1]$, one has $\dot{z}(t)\in{\rm conv}\{v_1,\cdots,v_k\}$. Then for every $\varepsilon>0$, there exists a piecewise affine path $\varsigma$ with $\parallel z-\varsigma\parallel_{W^{1,2}}<\varepsilon$, and so that $\dot{\varsigma}$ is composed of vectors from the set ${\rm conv}\{v_1,\cdots,v_k\}$, and $\varsigma(0)=z(0), \varsigma(1)=z(1)$.
\end{lemma}

The following is an analouge of \cite[Proposition~3.3]{PH19}.

\begin{proposition}\label{prop:int}
If a path $z \in H^1([0,1],\mathbb{R}^{2n})$ is such that $\dot{z}(t) = \sum_{i=1}^m \chi_{I_i}(t) w_i$ almost everywhere, where $\left(I_i = (\tau_{i-1},\tau_i) \right)_{i=1}^m$  is a partition of $[0,1]$, and $w_1,\cdots,w_m \in \mathbb{R}^{2n}$, then
\begin{equation}\label{e:3.1}
\int_0^1 \langle -J \dot{z}, z \rangle dt = \sum_{i=1}^m \sum_{j = 1}^{i-1} |I_j| |I_i| \omega_0(w_j,w_i)+\omega_0(z(0),z(1)).
\end{equation}
As usual $\sum_{j = 1}^{i-1} |I_j| |I_i| \omega_0(w_j,w_i)$ for $i=1$ is understood as zero.
\end{proposition}

\begin{proof}
The case $m=1$ is clear. Now we assume $m>1$. Since
$$
\int_0^1\langle -J\dot{z}(t),z(0)\rangle dt=-\langle J{z}(1),z(0)\rangle=-\omega_0(J{z}(1), Jz(0))=\omega_0(z(0), z(1))
$$
we deduce
\begin{eqnarray*}
\int_0^1 \langle -J \dot{z}, z \rangle dt  &=& \int_0^1 \langle -J \dot{z} , z(0) + \int_0^t \dot{z}(s) ds \rangle dt \\
 &=&\int_0^1\langle -J\dot{z}(t),z(0)\rangle dt \\
 &&+ \sum^m_{i=1}\int_{I_i}
  \langle -J \sum_{l=1}^m  \chi_{I_l}(t) w_l , \int_0^{\tau_{i-1}} \sum_{l=1}^m \chi_{I_l}(s) w_l ds + \int_{\tau_{i-1}}^t w_i ds \rangle dt \\
 &=& \omega_0(z(0),z(1))+\sum_{i=1}^m \int_{I_i} \langle -J w_i , \sum_{j<i}\int_{I_j} \sum_{l=1}^m \chi_{I_l}(s)w_l ds + (t-\tau_{i-1})w_i \rangle dt \\
 &=& \omega_0(z(0),z(1))+\sum_{i=1}^m \int_{I_i} \langle -J w_i , \sum_{j<i} \int_{I_j} w_j ds \rangle dt\\
 & =& \omega_0(z(0),z(1))+\sum_{i=1}^m \sum_{j<i}|I_i| |I_j| \omega_0(w_j,w_i).
\end{eqnarray*}
\end{proof}

Following the proof ideas of \cite[Lemma~3.1]{PH19} we can obtain:

\begin{lemma}\label{lem:1}
Given a set of vectors, $v_1, \ldots, v_k \in \mathbb{R}^{2n}$, for any piecewise affine path $z \in H^1([0,1],\R^{2n})$ with
 $\dot{z}(t) \in \text{conv}\{v_1,\ldots,v_{{k}}\}$ for almost every $t \in [0,1]$,  there exists another piecewise affine path $z' \in H^1([0,1],\R^{2n})$ so that $z'(0)=z(0),z'(1)=z(1),\dot{z}'(t) \in \{ v_1,\ldots,v_{{k}}\}$ for almost every $t$, and
$$ \int_0^1 \langle -J \dot{z}',z' \rangle dt \geq \int_0^1 \langle -J \dot{z}, z \rangle dt .$$
\end{lemma}

\begin{proof}
Write $\dot{z}(t) = \sum_{j=1}^m \chi_{I_j}(t) w_j$, where  $w_j \in \text{conv}\{v_1,\ldots,v_{{k}}\}$
for each $j$,  and $\left(I_j\right)_{j=1}^m$ is a partition of $[0,1]$.
Clearly, there exists $l=l(i)\in  \mathbb{N}$ such that
 $w_i = \sum_{j=1}^{l} a_{i_j} v_{i_j}$, where $a_{i_j} > 0$, $i_j \in \{1,\ldots,k\}$, and  $\sum_{j=1}^{l} a_{i_j} = 1$.
Consider the partition of $I_i$ to disjoint subintervals,
$\{I_{i_j}\}^l_{j=1}$, where the length of $I_{i_j}$ is $|I_{i_j}| = a_{i_j} |I_i|$.
Define
\begin{equation}\label{e:3.2}
\dot{z}_\ast(t) = \sum_{j<i} \chi_{I_j}(t) w_j+ \sum_{j=1}^l \chi_{I_{i_j}}(t) v_{i_j} + \sum_{j>i} \chi_{I_j}(t) w_j
\end{equation}
and $z_\ast(t)=z(0)+\int_0^t \dot{z}_\ast(s)ds$ for $t\in [0,1]$. Since $\int_0^1 \dot{z}_\ast(t)dt=\int_0^1\dot{z}(t)dt$, we deduce
 $z(0)=z_\ast(0)$ and $z(1)=z_\ast(1)$.
Then  Proposition~\ref{prop:int} leads to
$$\begin{aligned}
\int_0^1 \langle -J \dot{z}_\ast , z_\ast\rangle dt & = \omega_0(z_\ast(0),z_\ast(1))+\sum_{\substack{r < s \\ r,s \neq i}} |I_r| |I_s| \omega_0(w_r,w_s) + \sum_{j=1}^l \sum_{r<i} |I_r| |I_i| a_{i_j} \omega_0(w_r, v_{i_j})  \\
 & + \sum_{j=1}^l \sum_{r > i} |I_r| |I_i| a_{i_j} \omega_0(v_{i_j}, w_r) +  \sum_{1 \leq r < s \leq l} |I_i|^2 a_{i_r} a_{i_s} \omega_0(v_{i_r},v_{i_s}) \\
& =  \omega_0(z(0),z(1))+\sum_{\substack{r < s \\ r,s \neq i}} |I_r| |I_s| \omega_0(w_r,w_s) + \sum_{r<i} |I_r| |I_i| \omega_0(w_r, w_i) \\
& + \sum_{r>i} |I_r| |I_i| \omega_0(w_i, w_r) +  \sum_{1 \leq r < s \leq l} |I_i|^2 a_{i_r} a_{i_s} \omega_0(v_{i_r},v_{i_s}) \\
& = \int_0^1 \langle -J \dot{z} , z \rangle dt +  |I_i|^2\sum_{1 \leq r < s \leq l} a_{i_r} a_{i_s} \omega_0(v_{i_r},v_{i_s}).
\end{aligned}
$$
Define $b_{i_j}=a_{i_{l+1-j}}$ and $u_{i_j}=v_{i_{l+1-j}}$ for $j=1,\cdots,l$, and
$$
\hat{I}_j=I_j\;\hbox{for $j<i$ or $j>i$},\quad \hat{I}_{i_j}=I_{i_{l+1-j}}\;\hbox{for $j=1,\cdots,l$}.
$$
As above we may show that $z_{\ast\ast}(t)=z(0)+\int_0^t \dot{z}_{\ast\ast}(s)ds$ for $t\in [0,1]$, where
$$
\dot{z}_{\ast\ast}(t) = \sum_{j<i} \chi_{\hat{I}_j}(t) w_j + \sum_{j=1}^l \chi_{\hat{I}_{i_j}}(t) u_{i_j} + \sum_{j>i} \chi_{\hat{I}_j}(t) w_j,
$$
satisfies $z(0)=z_{\ast\ast}(0)$, $z(1)=z_{\ast\ast}(1)$ and
$$
\int_0^1 \langle -J \dot{z}_{\ast\ast}, z_{\ast\ast}\rangle dt=\int_0^1 \langle -J \dot{z} , z \rangle dt +  |I_i|^2\sum_{1 \leq r < s \leq l} b_{i_r} b_{i_s} \omega_0(u_{i_r},u_{i_s}).
$$
A straightforward computation as above gives rise to
$$
\sum_{1 \leq r < s \leq l} b_{i_r} b_{i_s}\omega_0(u_{i_r},u_{i_s})=-\sum_{1 \leq r < s \leq l} a_{i_r} a_{i_s} \omega_0(v_{i_r},v_{i_s}).
$$
Hence we can always choose $u\in\{z_\ast, z_{\ast\ast}\}$ so that
\begin{equation}\label{e:3.3}
\int_0^1 \langle -J \dot{u} , u\rangle dt \geq \int_0^1 \langle -J \dot{z}, z \rangle dt .
\end{equation}

Now starting from $z$ and choosing $i=1$ we get a path $z_1$ as above, Then starting from $z_1$ and choosing $i=2$ we get a path $z_2$ again.
Continuing this progress we obtain $z_1, z_2,\cdots, z_m$. Then $z':=z_m$ satisfies the requirements
of the lemma.
\end{proof}

Suitably modifying the proof of \cite[Lemma~3.5]{PH19}, we can get the following analogues of it.

\begin{lemma}\label{lem:2}
Given a finite sequence of pairwise distinct vectors $(v_1,\cdots,v_k)$,
 if $z \in H^1([0,1],\mathbb{R}^{2n})$ is a piecewise affine path such that $\dot{z}(t) = \sum_{i=1}^m \chi_{I_i}(t) w_i$
with $w_i \in \{v_1,\cdots,v_k\}$ for each $i$,  where $\left( I_i = (\tau_{i-1},\tau_i) \right)_{i=1}^m$ is a partition of $[0,1]$,
 then there exists another piecewise affine path $z'$ such that $\dot{z}'(t) \in \{v_1,\cdots,v_k\}$ for almost every $t$,
 $z'(0)=z(0),z'(1)=z(1)$, and $\{t : \dot{z}'(t) = v_j\}$ is connected for every $j = 1,\cdots,k$. In addition,
\begin{equation}\label{e:1}
 \int_0^1 \langle -J \dot{z}',z' \rangle dt \geq \int_0^1 \langle -J \dot{z}, z \rangle dt .
 \end{equation}
\end{lemma}

\begin{proof}
Assume $w_r=w_s$ for some $r<s$. Consider a rearrangement of the intervals $I_i$
 by deleting the intervals $I_s$ and increasing the length of the interval $I_r$ by $|I_s|=\tau_s-\tau_{s-1}$, that is,
$$I_i^\ast=\left\{
\begin{aligned}
&(\tau_{i-1},\tau_i), &i<r,\\
&(\tau_{i-1},\tau_i+\tau_s-\tau_{s-1}), &i=r,\\
&(\tau_{i-1}+\tau_s-\tau_{s-1},\tau_i+\tau_s-\tau_{s-1}), &r<i<s, \\
&\emptyset, &i=s, \\
&(\tau_{i-1},\tau_i), &i>s. \\
\end{aligned}
\right.
$$
Define $z_\ast$ by  $z_\ast(t)=z(0)+\int_0^t\dot{z}_\ast(s)ds$, where $\dot{z}_\ast(t)=\sum_{i=1}^m\chi_{I_i^\ast}(t)w_i$.
Then
$$
\int_0^1\dot{z}_\ast dt=\sum_{i=1}^m|I_i^\ast|w_i=\sum_{i=1}^m|I_i|w_i=\int_0^1\dot{z}dt
$$
and thus $z_\ast(0)=z(0)$ and $z_\ast(1)=z(1)$.
Since $I_i^\ast=I_i$ for $i<r$ or $i>s$, by Proposition~\ref{prop:int}, one can get
$$
\int_0^1\langle -J\dot{z}_\ast, z_\ast\rangle dt-\int_0^1\langle -J\dot{z},z\rangle dt= \sum_{i=r+1}^{s-1} 2|I_s||I_i|\omega_0(w_s,w_i).
 $$
Similarly,  by erasing $I_r$ and increasing the length of $I_s$ by $|I_r|$, we get a $z_{\ast\ast}$ such that
 $$
 \int_0^1\langle -J\dot{z}_{\ast\ast}, z_{\ast\ast}\rangle dt-\int_0^1\langle -J\dot{z},z\rangle dt= \sum_{i=r+1}^{s-1} 2|I_r||I_i|\omega_0(w_i,w_r).
  $$
It follows that either $z_\ast$ or $z_{\ast\ast}$ satisfies (\ref{e:1}).
Denote by $z_1\in \{z_\ast,z_{\ast\ast}\}$ satisfying (\ref{e:1}). Then
$$
z_1(t)=z(0)+\int_0^t\dot{z}_1(s)ds\quad\hbox{with}\;\dot{z}_1(t)=\sum_{i=1}^m\chi_{I_i^1}(t)w_i.
$$
Repeating this methods for different disjoint nonempty interval $I_r^1,I_s^1$ whenever $w_r=w_s$
we get a $z_2$ again. Proceeding with this progress for $z_2$, after finite steps we get a $z'$ with
the expected properties.
\end{proof}

Having the above lemmas we have the following corresponding result with \cite[Proposition~3.5]{PH19},
which may be proved by repeating the arguments therein because $H_K^*=\frac{1}{4}h_K^2$.

\begin{proposition}\label{prop:1}
 For a convex polytope  $K \subset \mathbb{R}^{2n}$ containing $0$ in the interior of it,
 let $\{F_i\}_{i=1}^{{\bf F}_{K}}$ be the $(2n-1)$-dimensional facets of it,
 let  $n_i$ be the unit outer normal to $F_i$,  let $p_i = J \partial H_K |_{F_i} = \frac{2}{h_i} J n_i$,
where $h_i := h_K(n_i)$ and $h_K(x)=\sup\{\langle y,x\rangle\,|\,y\in K\}$.
Let $c > 0$ be a constant and let $z \in H^1([0,1],\mathbb{R}^{2n})$  satisfies that for almost every $t$, there is a non-empty face of $K$, $F_{j_1} \cap \cdots \cap F_{j_l} \neq \emptyset$,  with $\dot{z}(t) \in c \cdot \text{conv}\{p_{j_1},\cdots,p_{j_l}\}$.
Then
$$
\int_0^1 H_K^*(-J\dot{z}(t))dt = c^2.
$$
\end{proposition}

\section{Proof of Theorem~\ref{th:main}}\label{sec:main}

We begin with a similar result to \cite[Theorem~1.5]{PH19}.

\begin{theorem}\label{th:4.1}
Let $K$ be a convex polytope as above (\ref{e:Comb formula}). Suppose $0\in{\rm Int}(K)$.
Then for any $\Psi\in{\rm Sp}(2n,\mathbb{R})$  there exists a generalized
$\Psi$-characteristic $\gamma: [0,1] \rightarrow \partial K $ with action
$$
A(\gamma)=\min\{A(x)>0\, |\, x \text{ is a generalized }  \Psi \text{-characteristic on } \partial K \}
$$
such that $\dot{\gamma}$ is piecewise constant and is composed of a finite sequence of vectors, i.e. there exists a sequence of vectors $(w_1,\ldots,w_m)$, and a sequence $(0=\tau_0<\cdots<\tau_{m-1}<\tau_{m}=1)$ so that $\dot{\gamma}(t) = w_i$ for $\tau_{i-1} < t < \tau_{i}$.
Moreover, for each $j \in \{1,\cdots,m\}$ there exists $i \in \{1,\cdots,{\bf F}_K\}$ so that $w_j = C_j J n_i$  for some $C_j > 0$, and for each $i \in \{1,\cdots,
{\bf F}_K\}$ and for every $C>0$ the set $\{t\in [0,1]\,|\, \dot{\gamma}(t) = C J n_i\}$ is either empty or connected, i.e. for every $i$ there is at most one $j \in \{1,\ldots,m\}$
 with $w_j = C_j J n_i$.
Hence $\dot{\gamma}$ has at most ${\bf F}_K$ discontinuous points, and $\gamma$ visits the interior of each facet at most once.
\end{theorem}

\begin{proof}
 Let $z:[0, T]\to\partial K$ be a generalized $\Psi$-characteristic with action $A(z)=c^\Psi_{\rm EHZ}(K)=T$.
 By Theorem~\ref{th:convex} we have $b\in{\rm Ker}(\Psi-I_{2n})$ and the $W^{1,2}$-path $u
 \in\mathcal{A}_\Psi$ satisfying $I_K(u)=T$ and (\ref{e:inclusion}). Thus we obtain
$\int_0^1 H_K^*(-J\dot{u}(t))dt = T$ by Proposition~\ref{prop:1}.
 For convenience let $c=T^{1/2}$.
  The next argument is the same as the proof of \cite[Theorem~1.5]{PH19}, we write it for completeness.

For every $N\in\mathbb{N}$,  Lemma~\ref{lem:PH3.1} yields a piecewise affine path $\zeta_N$ such that
$$
\parallel u-\zeta_N\parallel_{W^{1,2}}\leqslant {\frac{1}{N}}\quad\hbox{and}\quad
 \dot{\zeta}_N(t)\in c \cdot {\rm conv}\{p_1,\cdots,p_{{\bf F}_K}\}
 $$
 for almost every $t$, $\zeta_N(0)=u(0),\zeta_N(1)=u(1)$. By applying Lemma~\ref{lem:1} with $v_i=cp_i, i=1,\cdots,{\bf F}_K$ to $\zeta_N$,
we get a piecewise affine path $\zeta_N' \in W^{1,2}([0,1],\R^{2n})$ such that
$$
\zeta_N'(0)=u(0),\;\; \zeta_N'(1)=u(1),\;\;
\dot{\zeta}_N'(t) \in \{ v_1,\ldots, v_{{\bf F}_K}\} \;\;{\rm a.e.},\;\;\hbox{and}\; A(\zeta_N')\ge A(\zeta_N).
$$
Applying Lemma~\ref{lem:2}  to $\zeta_N'$ again, we get a piecewise affine path $u_N:[0,1]\to\R^{2n}$
from $u(0)$ to $u(1)$ such that
$$
\dot{u}_N(t)=\sum_{i=1}^{m_N}\chi_{I_i^N}(t)v_i^N
$$
where $v_i^N=v_j$ for some $j\in\{1,\cdots,{\bf F}_K\}$ and for every $j$ there is at most one such $i$,
  and that
$$
A_N:=\sqrt{A(u_N)}\geqslant \sqrt{A(\zeta_N)}.
$$
Define $u_N':=\frac{u_N}{A_N}\in\mathcal{A}_\Psi$  and  $c_N=:\frac{c}{A_N}$.
Write  $w_i^N:=\frac{v_i^N}{A_N}$ for the velocities of $u_N'$, which sits in
the set $\frac{c}{A_N}\cdot\{p_1,\cdots,p_{{\bf F}_K}\}$. Since
$\parallel u-\zeta_N\parallel_{W^{1,2}}\leqslant {\frac{1}{N}}$ we deduce  that $A(\zeta_N)\rightarrow 1$
as $N\rightarrow\infty$. Hence $\varliminf_{N\rightarrow\infty} A_N\geqslant 1$, and $\varlimsup_{N\rightarrow\infty}c_N\leqslant c$.
Moreover  Proposition~\ref{prop:1} and  the minimality of $I_K(u)$
imply $c_N^2=I_K(u_N')\geqslant I_K(u)=c^2$. We deduce $\lim_{N\rightarrow\infty}c_N=c$ and thus
$\lim_{N\rightarrow\infty}A_N=1$.

Let $\mathcal{A}^1$ consist of $z\in H^1([0,1],\mathbb{R}^{2n})$ for which there exist $C>0$ and
an increasing sequence of numbers $0 = \tau_0 \leq \tau_1 \leq \ldots \leq \tau_{{\bf F}_K} = 1$
such that
$$
\dot{z}(t)=\sum_{i=1}^{{\bf F}_K}\chi_{I_i}(t) C\cdot p_{\sigma(i)}
$$
with $I_i = (\tau_{i-1},\tau_i)$, where $\sigma\in S_{{\bf F}_K}$ is the permutations on $\{1,\cdots,{\bf F}_K\}$.
 Define a map
 \begin{equation}\label{e:Phi}
 \Phi:\mathcal{A}^1\rightarrow S_{{\bf F}_K}\times\mathbb{R}^{{\bf F}_K},\; z\mapsto (\sigma,(|I_1|,\cdots,|I_{{\bf F}_K}|)).
 \end{equation}
Clearly, the image ${\rm Im}(\Phi)$ is contained in the compact subset of
$S_{{\bf F}_K}\times\mathbb{R}^{{\bf F}_K}$,
$$
S_{{\bf F}_K}\times\left\{(t_1,\cdots,t_{{\bf F}_K})\in\mathbb{R}^{{\bf F}_K}\,\bigg|\, t_i\geqslant0\;\forall i,\;\sum_{i=1}^{{\bf F}_K}t_i=1\right\}.
$$
Since  $u_N'\in\mathcal{A}^1$ with $C=c_N$, we can write $\Phi(u_N')=(\sigma^N,(t_1^N,\cdots, t_{{\bf F}_K}^N))$.
After passing to a subsequence, we can assume that $\sigma^N=\sigma$ is constant, and $(t_1^N,\cdots, t_{{\bf F}_K}^N)$ converges to a vector $(t_1^{\infty},\cdots, t_{{\bf F}_K}^{\infty})$.
Define
\begin{eqnarray*}
&&\tau_0^\infty=0,\; \tau_1^\infty=\tau_0^\infty+t_1^{\infty},
 \;\tau_j^\infty=\tau_0^\infty+ \sum^j_{i=1}t_i^{\infty},\;j=2,\cdots, {{\bf F}_K},\\
&&I_i^\infty = (\tau_{i-1}^\infty,\tau_i^\infty),\; i=1,\cdots, {{\bf F}_K}
\end{eqnarray*}
and the piecewise affine path  $u_{\infty}'(t):=u(0)+\int^t_0\dot{u}_{\infty}'(s)ds$ with
$$
\dot{u}_{\infty}'(t)=\sum_{i=1}^{{\bf F}_K}\chi_{I_i^\infty}(t) c\cdot p_{\sigma(i)}.
$$
Let $\mathcal{T}^N=\{t\in [0,1]\,|\, \dot{u}_N'(t)=\frac{c}{c_N}\dot{u}_{\infty}'(t)\}$. Then
$$
\int_{\mathcal{T}^N}\parallel \dot{u}_N'(t)-\dot{u}_{\infty}'(t)\parallel^2dt\rightarrow 0\quad\hbox{as $N\rightarrow\infty$}.
$$
Since
$\parallel \dot{u}_N'(t)-\dot{u}_{\infty}'(t)\parallel^2$ is bounded on
$\{t\in[0,1]\,|\, \hbox{$\dot{z}_N'(t)$ and $\dot{z}_{\infty}'(t)$ are defined}\}$,
  as $N\rightarrow\infty$ we get $|\mathcal{T}^N|\rightarrow 1$ and therefore
  $$
  \int_{[0,1]\setminus\mathcal{T}^N}\parallel \dot{u}_N'(t)-\dot{u}_{\infty}'(t)\parallel^2dt\rightarrow 0.
  $$
Observe that $\lim_{N\rightarrow\infty}\int_0^1 \dot{u}_N'(t)dt=\int_0^1\dot{u}(t)dt$ implies $\int_0^1\dot{u}_{\infty}'(t)dt=\int_0^1\dot{u}(t)dt$. We deduce
$$
u_{\infty}'(1)=u_{\infty}'(0)+\int_0^1\dot{u}_{\infty}'(t)dt=u(0)+\int_0^1\dot{u}(t)dt=u(1)
$$
and so $u_{\infty}'(1)=\Psi u_{\infty}'(0)$. Moreover
$$\begin{aligned}
 |A(u_{\infty}')-1|&=|A(u_{\infty}')-A(u_N')|\\
 &=\left|\frac{1}{2}\int_0^1 \langle -J\dot{u}_{\infty}'(t), u_{\infty}'(t)\rangle-\langle
 -J\dot{u}_N'(t), u_N'(t)\rangle dt\right|\\
 &\leqslant\left|\frac{1}{2}\int_0^1\langle -J(\dot{u}_{\infty}'(t)-\dot{u}_N'(t)),
 u_{\infty}'(t)\rangle
 dt\right|+\left|\frac{1}{2}\int_0^1\langle-J\dot{u}_N'(t), u_{\infty}'(t)-u_N'(t)\rangle dt\right|\\
 &\leqslant\frac{1}{2}\int_0^1|\dot{u}_{\infty}'(t)-\dot{u}_N'(t)||u_{\infty}'(t)|dt+\frac{1}{2}\int_0^1|\dot{u}_N'(t)||u_{\infty}'(t)-u_N'(t)|dt\rightarrow
 0
 \end{aligned}$$
 because  $\dot{u}_N'$ and $u_{\infty}'$ are bounded. Then
 $A(u_{\infty}')=1$,  and thus
 $u_{\infty}'\in\mathcal{A}_\Psi$ and
 $$
 I_K(u_{\infty}')=
 \lim_{N\to\infty}I_K(u_N')=\lim_{N\to\infty}c_N^2=
 c^2=T=c^\Psi_{\rm EHZ}(K).
 $$
 By Theorem~\ref{th:convex} we have $a_0\in {\rm Ker}(\Psi-I_{2n})$  such that the $W^{1,2}$-path
\begin{equation}\label{e:repara11}
[0, T]\ni t\mapsto \gamma^\ast(t)=\sqrt{T}u'_\infty(t/T)+a_0/\sqrt{T}
\end{equation}
is a piecewise affine generalized $\Psi$-characteristic  on $\partial K$
with action $A(\gamma^\ast)=c^\Psi_{\rm EHZ}(K)$. Then  the generalized $\Psi$-characteristic  on $\partial K$,
$[0,1]\ni t\mapsto \gamma(t):=\gamma^\ast(Tt)$, has
 action $A(\gamma)=c^\Psi_{\rm EHZ}(K)$ and satisfies
   $\dot{\gamma}(t)\in T\cdot\{p_1,\cdots,p_{{\bf F}_K}\}$ for almost every $t\in [0,1]$ and that
    the set $\{t:\dot{\gamma}(t)=p_i\}$ is connected for every $i$.
Recall $p_i=\frac{2}{h_i}Jn_i$. Theorem~\ref{th:4.1} is proved.
\end{proof}

\begin{proof}[Proof of Theorem \ref{th:main}]
{\bf Step 1}. {\it Case $0\in{\rm Int}(K)$}.
Let $\mathcal{A}_\Psi^0$ consist of $z\in \mathcal{A}_\Psi$ for which there exist $C>0$ and
an increasing sequence of numbers $0 = \tau_0 \leq \tau_1 \leq \ldots \leq \tau_{{\bf F}_K} = 1$
such that
\begin{equation}\label{e:affine}
\dot{z}(t)=\sum_{i=1}^{{\bf F}_K}\chi_{I_i}(t) C\cdot p_{\sigma(i)}
\end{equation}
with $I_i = (\tau_{i-1},\tau_i)$, where $\sigma\in S_{{\bf F}_K}$ is a permutation on $\{1,\cdots,{\bf F}_K\}$.
Then $u'_\infty$ in the proof of Theorem~\ref{th:4.1} belongs to $\mathcal{A}_\Psi^0$
and satisfies $I_K(u'_\infty)=c^{\Psi}_{{\rm EHZ}}(K)$.
Thus
\begin{equation}\label{e:affine1}
c^\Psi_{\rm EHZ}(K)=\min\{I_K(z)\,|\,z\in \mathcal{A}_\Psi\}=\min\{I_K(z)\,|\,z\in \mathcal{A}_\Psi^0\}.
\end{equation}
For any $z\in \mathcal{A}_\Psi^0$, $\dot{z}$ has the form of (\ref{e:affine}) and hence
$$
z(1)-z(0)=\int_0^1\dot{z}(t)
dt=C\sum_{i=1}^{{\bf F}_K}T_i p_{\sigma(i)}
$$
where $T_i=|I_i|$, and Proposition~\ref{prop:int} yields
$$
1=\frac{1}{2}\int_0^1 \langle -J \dot{z}, z \rangle dt=\frac{1}{2}C^2\sum_{1\leqslant
j<i\leqslant{\bf F}_K} T_iT_j\omega_0(p_{\sigma(j)},p_{\sigma(i)})+\frac{1}{2}\omega_0(z(0),z(1)).
$$
Let $v=z(0)/C$. The above two formulas become, respectively,
$\Psi v-v=\sum_{i=1}^{{\bf F}_K}T_i p_{\sigma(i)}$ and
$$
1=\frac{1}{2}\int_0^1 \langle -J \dot{z}, z \rangle dt=\frac{1}{2}C^2\sum_{1\leqslant
j<i\leqslant{\bf F}_K} T_iT_j\omega_0(p_{\sigma(j)},p_{\sigma(i)})+ C^2\frac{1}{2}\omega_0(v,\Psi v).
$$
By Proposition~\ref{prop:1} we have
 $I_K(z)=C^2$, and thus
 \begin{equation}\label{e:affine*}
 I_K(z)=\frac{2}{\sum_{1\leqslant j<i\leqslant{\bf F}_K}
T_iT_j\omega_0(p_{\sigma(j)},p_{\sigma(i)})-\omega_0(\Psi v, v)}>0.
\end{equation}
 With $E_{\Psi}$ defined as in Theorem~\ref{th:main} let
$$
M^\ast_{\Psi}(K)=\left\{((T_i)_{i=1}^{{\bf F}_K},v,\sigma)\,\bigg|\,\begin{array}{ll}
&\sigma\in
S_{{\bf F}_K},\; T_i\geqslant
0,\sum_{i=1}^{{\bf F}_K}T_i=1,\sum_{i=1}^{{\bf F}_K}T_ip_{\sigma(i)}=\Psi v-v,\\
&\sum_{1\leqslant j<i\leqslant{\bf F}_K}T_iT_j\omega_0(p_{\sigma(j)},p_{\sigma(i)})>\omega_0(\Psi v, v),\;v\in E_{\Psi}
\end{array}
\right\},
$$
For every triple  $((T_i)_{i=1}^{{\bf F}_K}, v, \sigma)\in M^\ast_\Psi(K)$, as the construction of $u'_\infty$ in the proof of Theorem~\ref{th:4.1}
we can use it to construct a $z\in \mathcal{A}_\Psi^0$ such that (\ref{e:affine*}) holds.
It follows from these and (\ref{e:affine1}) that
$$
c^{\Psi}_{{\rm EHZ}}(K)=\min_{((T_i)_{i=1}^{{\bf F}_K},v, \sigma)\in M^\ast_\Psi(K)}\frac{2}{\sum_{1\leqslant j<i\leqslant{\bf F}_K}
T_iT_j\omega_0(p_{\sigma(j)},p_{\sigma(i)})-\omega_0(\Psi v, v)},
$$
Let $\beta_{\sigma(i)}=\frac{T_i}{h_{\sigma(i)}}$. Since $p_i=\frac{2}{h_i}Jn_i$, we get
$$
c^{\Psi}_{{\rm EHZ}}(K)=\min_{((\beta_i)_{i=1}^{{\bf F}_K},v, \sigma)\in M_{\Psi}(K)}\frac{2}{4\sum_{1\leqslant j<i\leqslant{\bf F}_K}
\beta_{\sigma(i)}\beta_{\sigma(j)}\omega_0(n_{\sigma(j)},n_{\sigma(i)})-\omega_0(\Psi v, v)},
$$
where $M_{\Psi}(K)$ is as in Theorem~\ref{th:main}.

\noindent{\bf Step 2}. {\it General case}.
 Let $p\in{\rm Int}(K)$ be a fixed point of $\Psi$.
  Consider the symplectomorphism
  \begin{equation}\label{e:translate}
\phi:(\mathbb{R}^{2n},\omega_0)\rightarrow (\mathbb{R}^{2n},\omega_0), x\mapsto x-p.
\end{equation}
Since $\Psi(p)=p$,  $\phi\circ\Psi=\Psi\circ\phi$ and thus $c_{{\rm EHZ}}^{\Psi}(K)=c_{{\rm EHZ}}^{\Psi}(\phi(K))$
by the arguments below Proposition~1.2 of \cite{JinLu1916}.
 Let us write $\hat{K}=\phi(K)$ for convenience.
 Denote all $(2n-1)$-dimensional facets of it by  $\{\hat{F}_i\}_{i=1}^{{\bf F}_{\hat{K}}}$, the unit outer normal to $\hat{F}_i$ by  $\hat{n}_i$,
 the support function of $\hat{K}$ by $h_{\hat{K}}$. Then ${\bf F}_{\hat{K}}={\bf F}_{{K}}$,
 $\hat{F}_i=F_i-p$ and $\hat{n}_i=n_i$ for $i=1,\cdots, {\bf F}_{{K}}$, and
 $h_{\hat{K}}(y)=h_K(y)-\langle p,y\rangle$. By Step 1 we get
  $$
c^{\Psi}_{{\rm EHZ}}(\hat{K})=\min_{\bigl((\beta_i)_{i=1}^{{\bf F}_K},v,\sigma\bigr)\in M_{\Psi}(\hat{K})}\frac{2}{4\sum_{1\leqslant j<i\leqslant{\bf F}_K}
\beta_{\sigma(i)}\beta_{\sigma(j)}\omega_0(n_{\sigma(j)},n_{\sigma(i)})-\omega_0(\Psi v, v)},
$$
where with $\hat{h}_i=\hat{h}_{\hat{K}}(n_i)=h_K(n_i)-\langle p,n_i\rangle=h_i-\langle p,n_i\rangle$ for $i=1,\cdots, {\bf F}_{{K}}$,
$$
M_{\Psi}(\hat{K})=\left\{\big((\beta_i)_{i=1}^{{\bf F}_K},v, \sigma\bigr)\,\bigg|\,\begin{array}{ll}
&\sigma\in
S_{{\bf F}_K},\;\beta_i\geqslant
0,\;\sum_{i=1}^{{\bf F}_K}\beta_i\hat{h}_i=1,\;\sum_{i=1}^{{\bf F}_K}2\beta_i Jn_i=\Psi v-v, \\
&4\sum_{1\leqslant j<i\leqslant{\bf F}_K}
\beta_{\sigma(i)}\beta_{\sigma(j)}\omega_0(n_{\sigma(j)},n_{\sigma(i)})>\omega_0(\Psi v, v),\;v\in
E_{\Psi}
\end{array}
\right\}.
$$
Clearly, it remains to prove $M_{\Psi}(\hat{K})=M_{\Psi}({K})$. In fact, for any $\big((\beta_i)_{i=1}^{{\bf F}_K},v,\sigma\bigr)\in M_{\Psi}(\hat{K})$,
since
$$
1=\sum_{i=1}^{{\bf F}_K}\beta_i\hat{h}_i=\sum_{i=1}^{{\bf F}_K}\beta_i{h}_i-\langle p, \sum_{i=1}^{{\bf F}_K}\beta_in_i\rangle,
$$
 it suffices to prove $\langle p, \sum_{i=1}^{{\bf F}_K}\beta_in_i\rangle=0$.
 Note that $\sum_{i=1}^{{\bf F}_K}2\beta_i Jn_i=\Psi v-v$, $v\in E_{\Psi}$.
 We have
 \begin{eqnarray*}
 \langle p, \sum_{i=1}^{{\bf F}_K}\beta_in_i\rangle&=&\omega_0(p, \sum_{i=1}^{{\bf F}_K}\beta_i Jn_i)=
 \frac{1}{2}\omega_0(p, \Psi v-v)= \frac{1}{2}(\omega_0(p, \Psi v)-\omega_0(p,v))=0
 \end{eqnarray*}
 because $\omega_0(p, \Psi v)=\omega_0(\Psi p, \Psi v)=\omega_0(p, v)$. Hence
 $M_{\Psi}(\hat{K})\subset M_{\Psi}({K})$, and hence $M_{\Psi}({K})\subset M_{\Psi}(\hat{K})$
  since $K=\hat{K}-(-p)$ and $\Psi(-p)=-p$.
 \end{proof}

\section{Proofs of Theorems~\ref{th:main2},~\ref{th:converse}}\label{sec:main-con}

We have an analogue of Theorem~\ref{th:4.1}:
\begin{theorem}\label{th:5.1}
Let $K$ be a convex polytope as above (\ref{e:Comb formula}). If $0\in {\rm Int}(K)$,
 there exists a generalized leafwise chord on $\partial K \text{ for }\mathbb{R}^{n,k}$: $\gamma: [0,1] \rightarrow \partial K $ with $A(z)=\min\{A(x)|x $ is a generalized leafwise chord on $\partial K \text{ for }\mathbb{R}^{n,k}\}$ such that $\dot{\gamma}$ is piecewise constant and is composed of a finite sequence of vectors, i.e. there exists a sequence of vectors $(w_1,\ldots,w_m)$, and a sequence $(0=\tau_0<\cdots<\tau_{m-1}<\tau_{m}=1)$ so that $\dot{\gamma}(t) = w_i$ for $\tau_{i-1} < t < \tau_{i}$.
Moreover, for each $j \in \{1,\cdots,m\}$ there exists $i \in \{1,\cdots,{\bf F}_K\}$ so that $w_j = C_j J n_i$ , for some $C_j > 0$, and for each $i \in \{1,\cdots,{\bf F}_K\}$, the set $\{t : \exists C>0, \dot{\gamma}(t) = C J n_i\}$ is connected, i.e. for every $i$ there is at most one $j \in \{1,\ldots,m\}$ with $w_j = C_j J n_i$.
Hence there are at most ${\bf F}_K$ points of discontinuity in $\dot{\gamma}$, and $\gamma$ visits the interior of each facet at most once.
\end{theorem}
\begin{proof}

Let $z:[0,T]\to \partial K$ be a generalized leafwise chord  with action $A(z)=c_{{\rm LR}}(K, K\cap\mathbb{R}^{n,k})=c^{n,k}(K)$  for $\mathbb{R}^{n,k}$.
By Theorem~\ref{th:cosi} we can assume it to satisfy (\ref{e:repara4}) (by a reparametrization if necessary), and obtain that
the path
$$
u:[0,1]\to\mathbb{R}^{2n},\;t\mapsto \frac{1}{\sqrt{T}}z(Tt)- \frac{1}{\sqrt{T}}P_{n,k}\int^1_0 z(Tt)dt
$$
 belongs to  $\mathcal{A}_2$ and satisfies $I_2(u)=T=c^{n,k}(K)$.  Moreover
$$
\dot{u}(t)=\sqrt{T}\dot{z}(Tt)\in \sqrt{T}{\rm conv}\{p_i\,|\, \sqrt{T}(u(t)-b)\in F_i\}\subset T^{1/2}\cdot {\rm conv}\{p_1,\cdots,p_{{\bf F}_K}\}
$$
with $b=- \frac{1}{\sqrt{T}}P_{n,k}\int^1_0 z(Tt)dt$ and with $c=T^{1/2}$, and so  $I_2(u)=c^2$ by Proposition \ref{prop:1}.

For every $N\in\mathbb{N}$,  Lemma~\ref{lem:PH3.1} yields a piecewise affine path $\zeta_N$ such that
$$
\parallel u-\zeta_N\parallel_{W^{1,2}}\leqslant {\frac{1}{N}},\quad \zeta_N(0)=u(0),\quad \zeta_N(1)=u(1)\quad\hbox{and}\quad
\dot{\zeta}_N(t)\in c \cdot {\rm conv}\{p_1,\cdots,p_{{\bf F}_K}\}
$$
for almost every $t$. By applying Lemma~\ref{lem:1} with $v_i=cp_i, i=1,\cdots,{\bf F}_K$ to $\zeta_N$,
we get a piecewise affine path $\zeta_N' \in W^{1,2}([0,1],\R^{2n})$ such that
$$
A(\zeta_N')\ge A(\zeta_N),\quad \zeta_N'(0)=u(0),\quad \zeta_N'(1)=u(1),\quad
\dot{\zeta}_N'(t) \in \{ v_1,\ldots, v_{{\bf F}_K}\}
$$
for almost every $t$.
Applying Lemma~\ref{lem:2}  to $\zeta_N'$ again, we can obtain a piecewise affine path $u_N:[0,1]\to\R^{2n}$
from $u(0)$ to $u(1)$ such that
$$
\dot{u}_N(t)=\sum_{i=1}^{m_N}\chi_{I_i^N}(t)v_i^N
$$
where $v_i^N=v_j$ for some $j\in\{1,\cdots,{\bf F}_K\}$ and for every $j$ there is at most one such $i$,
  and that
$$
A_N:=\sqrt{A(u_N)}\geqslant \sqrt{A(\zeta_N)}.
$$
Define $u_N':=\frac{u_N}{A_N}$  and  $c_N=:\frac{c}{A_N}$. Notice that
$\int_0^1u_N'(t)dt$ may not belong to $JV_0^{n,k}$ and $u_N'$ may not belong to $\mathscr{F}_2$.
Recall that  $P_{n,k}:\mathbb{R}^{2n}=JV_0^{n,k}\oplus\mathbb{R}^{n,k}\to \mathbb{R}^{n,k}$ is the orthogonal projection. Define
$$
y_N:=u_N'-P_{n,k}\left(\int_0^1u_N'(t)dt\right).
$$
Then $\int_0^1y(t)dt\in JV_0^{n,k}$ and
\begin{eqnarray*}
A(y_N)&=&\int_0^1\bigg\langle -J\dot{u}_N', u_N'(t)-P_{n,k}\Big(\int_0^1u_N'(t)dt\Big)\bigg\rangle dt\\
&=&A(u_N')-\Big\langle J(u_N'(1)-u_N'(0)), P_{n,k}\big(\int_0^1u_N'(t)dt\big)\Big\rangle.
\end{eqnarray*}
Since $u_N'(1)-u_N'(0)\in V_0^{n,k}$,  $A(y_N)=A(u_N')=1$. Thus, $y_N\in\mathcal{A}_2$.
Write  $w_i^N:=\frac{v_i^N}{A_N}$ for the velocities of $y_N$, which sits in
the set $\frac{c}{A_N}\cdot\{p_1,\cdots,p_{{\bf F}_K}\}$. Since
$\parallel u-\zeta_N\parallel_{W^{1,2}}\leqslant {\frac{1}{N}}$ we deduce  that $A(\zeta_N)\rightarrow 1$
as $N\rightarrow\infty$. Hence $\varliminf_{N\rightarrow\infty} A_N\geqslant 1$, and $\varlimsup_{N\rightarrow\infty}c_N\leqslant c$.
Moreover  Proposition~\ref{prop:1} and  the minimality of $I_2(u)$
imply that $c_N^2=I_K(y_N)\geqslant I_K(u)=c^2$. Then $\lim_{N\rightarrow\infty}c_N=c$ and thus
$\lim_{N\rightarrow\infty}A_N=1$.

Recall that the set $\mathcal{A}^1$ is defined as above (\ref{e:Phi})
and that the map $\Phi$ is as in (\ref{e:Phi}).
%
%Recall in the proof of Theorem~\ref{th:4.1} that $\mathcal{A}^1$ consists of $z\in H^1([0,1],\mathbb{R}^{2n})$ for which there exist $C>0$ and
%an increasing sequence of numbers $0 = \tau_0 \leq \tau_1 \leq \ldots \leq \tau_{{\bf F}_K} = 1$
%such that
%$$
%\dot{z}(t)=\sum_{i=1}^{{\bf F}_K}\chi_{I_i}(t) C\cdot p_{\sigma(i)}
%$$
%with $I_i = (\tau_{i-1},\tau_i)$, where $\sigma\in S_{{\bf F}_K}$ is the permutations on $\{1,\cdots,{\bf F}_K\}$,
% and that the map
% $$
% \Phi:\mathcal{A}^1\rightarrow S_{{\bf F}_K}\times\mathbb{R}^{{\bf F}_K}, z\mapsto (\sigma,(|I_1|,\cdots,|I_{{\bf F}_K}|)).
% $$
By the proof of Theorem~\ref{th:4.1}, the image ${\rm Im}(\Phi)$ is contained in the compact subset of
$S_{{\bf F}_K}\times\mathbb{R}^{{\bf F}_K}$,
$$
S_{{\bf F}_K}\times\{(t_1,\cdots,t_{{\bf F}_K})\in\mathbb{R}^{{\bf F}_K}\,|\, t_i\geqslant0\;\forall i,\;\sum_{i=1}^{{\bf F}_K}t_i=1\}.
$$
Since  $y_N\in\mathcal{A}^1$ with $C=c_N$, we can write $\Phi(y_N)=(\sigma^N,(t_1^N,\cdots, t_{{\bf F}_K}^N))$.
After passing to a subsequence, we can also assume that $\sigma^N=\sigma$ is constant, and $(t_1^N,\cdots, t_{{\bf F}_K}^N)$ converges to a vector $(t_1^{\infty},\cdots, t_{{\bf F}_K}^{\infty})$.
Define
\begin{eqnarray*}
&&\tau_0^\infty=0,\quad \tau_1^\infty=\tau_0^\infty+t_1^{\infty},\quad
 \tau_j^\infty=\tau_0^\infty+ \sum^j_{i=1}t_i^{\infty},\quad j=2,\cdots, {{\bf F}_K},\\
&&I_i^\infty = (\tau_{i-1}^\infty,\tau_i^\infty),\quad i=1,\cdots, {{\bf F}_K}
\end{eqnarray*}
and  the piecewise affine path  $u_{\infty}'(t)=u(0)+\int^t_0\dot{u}_{\infty}'(s)ds$ with
$$
\dot{u}_{\infty}'(t)=\sum_{i=1}^{{\bf F}_K}\chi_{I_i^\infty}(t) c\cdot p_{\sigma(i)}.
$$

Similar to the proof of Theorem~\ref{th:4.1}, one gets $u_{\infty}'$ satisfying $u_{\infty}'(0)=u(0), u_{\infty}'(1)=u(1)$, $ A(u_{\infty}')=1$
and $I_{2}(u_{\infty}')=c^2$.
Define
$$
u_{\infty}:=u_{\infty}'-P_{n,k}\bigg(\int_0^1u_{\infty}'(t)dt\bigg).
$$
Then $u_{\infty}\in\mathcal{A}_2$ and $I_2(u_{\infty})=T=c^{n,k}(K)$.
By Theorem~\ref{th:cosi} we have
${\bf a}_0\in \mathbb{R}^{n,k}$ such that
$$
[0, 1]\ni t\mapsto \gamma(t):=\sqrt{T}u_\infty(t)+ {\bf a}_0/\sqrt{T}
$$
is a  piecewise affine generalized leafwise chord  on $\partial K$  for $\mathbb{R}^{n,k}$ with action
 $$
 A(\gamma)=I_2(u)=c_{{\rm LR}}(K, K\cap\mathbb{R}^{n,k})
 $$
  and satisfying  $\dot{\gamma}(t)\in T\cdot\{p_1,\cdots,p_{{\bf F}_K}\}$ for almost every $t\in [0,1]$ and that
    the set $\{t:\dot{\gamma}(t)=p_i\}$ is connected for every $i$.
Recall $p_i=\frac{2}{h_i}Jn_i$. Theorem~\ref{th:5.1} is proved.
\end{proof}

\begin{proof}[Proof of Theorem~\ref{th:main2}]
{\bf Step 1}. {\it Case $0\in{\rm Int}(K)$}.
Let $\mathcal{A}_2^0$ consist of $z\in \mathcal{A}_2$ for which there exist $C>0$ and
an increasing sequence of numbers $0 = \tau_0 \leq \tau_1 \leq \ldots \leq \tau_{{\bf F}_K} = 1$
such that
\begin{equation}\label{e:affine2}
\dot{z}(t)=\sum_{i=1}^{{\bf F}_K}\chi_{I_i}(t) C\cdot p_{\sigma(i)}
\end{equation}
with $I_i = (\tau_{i-1},\tau_i)$, where $\sigma\in S_{{\bf F}_K}$ is the permutation on $\{1,\cdots,{\bf F}_K\}$.
Then $u'_\infty$ in the proof of Theorem~\ref{th:5.1} belongs to $\mathcal{A}_2^0$
and satisfies $I_K(u'_\infty)=c_{{\rm LR}}(K, K\cap\mathbb{R}^{n,k})$.
Thus
\begin{equation}\label{e:affine3}
c_{{\rm LR}}(K, K\cap\mathbb{R}^{n,k})=\min\{I_2(z)\,|\,z\in \mathcal{A}_2\}=\min\{I_2(z)\,|\,z\in \mathcal{A}_2^0\}.
\end{equation}
For any $z\in \mathcal{A}_2^0$, we have $z(0),z(1)\in\mathbb{R}^{n,k}$, $\dot{z}$ has the form of (\ref{e:affine2})
and hence
$$
V_0^{n,k}\ni z(1)-z(0)=\int_0^1\dot{z}(t)
dt=C\sum_{i=1}^{{\bf F}_K}T_i p_{\sigma(i)}
$$
where $T_i=|I_i|$, and Proposition~\ref{prop:int} yields
$$
1=\frac{1}{2}\int_0^1 \langle -J \dot{z}, z \rangle dt=\frac{1}{2}C^2\sum_{1\leqslant
j<i\leqslant{\bf F}_K} T_iT_j\omega_0(p_{\sigma(j)},p_{\sigma(i)})+\frac{1}{2}\omega_0(z(0),z(1)).
$$
Note that $\omega_0(z(0),z(1))=\omega_0(z(0),z(1)-z(0))=0$, and
$I_2(z)=C^2$ by Proposition~\ref{prop:1}. Then
\begin{equation}\label{e:affine4}
I_2(z)=\frac{2}{\sum_{1\leqslant j<i\leqslant{\bf F}_K}
T_iT_j\omega_0(p_{\sigma(j)},p_{\sigma(i)})}>0.
\end{equation}
Let
$$
M^\ast(K)=\left\{((T_i)_{i=1}^{{\bf F}_K},\sigma)\,\bigg|\,\begin{array}{ll}
&\sigma\in S_{{\bf F}_K}, T_i\geqslant
0,\,\sum_{i=1}^{{\bf F}_K}T_i=1,\,\sum_{i=1}^{{\bf F}_K}T_ip_{\sigma(i)}\in V_0^{n,k}\\
&\sum_{1\leqslant j<i\leqslant{\bf F}_K} T_iT_j\omega_0(p_{\sigma(j)},p_{\sigma(i)})>0
\end{array}
\right\}.
$$
For every pair $((T_i)_{i=1}^{{\bf F}_K},\sigma)\in M^\ast(K)$, as in the construction of $u'_\infty$ in the proof of Theorem~\ref{th:4.1}
we can use $((T_i)_{i=1}^{{\bf F}_K},\sigma)$ to construct a $z\in \mathcal{A}_2^0$ such that (\ref{e:affine4}) holds.
It follows  that
$$
c_{{\rm LR}}(K, K\cap\mathbb{R}^{n,k})=\min_{((T_i)_{i=1}^{{\bf F}_K},\sigma)\in M^\ast(K)}\frac{2}{\sum_{1\leqslant
j<i\leqslant{\bf F}_K} T_iT_j\omega_0(p_{\sigma(j)},p_{\sigma(i)})},
$$
Define $\beta_{\sigma(i)}:=\frac{T_i}{h_{\sigma(i)}}$. Since $p_i=\frac{2}{h_i}Jn_i$,  The above two formulas give the desired formula in this case.

\noindent{\bf Step 2}. {\it General case}.
 Let $p\in{\rm Int}(K)\cap\mathbb{R}^{n,k}$. Then the symplectomorphism $\phi$ defined by
 (\ref{e:translate}) satisfies
$c_{{\rm LR}}(\phi(K), \phi(K)\cap\mathbb{R}^{n,k})=c_{{\rm LR}}(K, K\cap\mathbb{R}^{n,k})$
by the arguments at the beginning of \cite[\S3]{JinLu1917}.
As in Step~2 of the proof of Theorem \ref{th:main}
let $\hat{K}=\phi(K)$.
 By Step 1 we obtain
$$
c_{{\rm LR}}(\hat{K}, \hat{K}\cap\mathbb{R}^{n,k})=\frac{1}{2}\min_{
((\beta_i)_{i=1}^{{\bf F}_K},\sigma)\in M(\hat{K})}\frac{1}{\sum_{1\leqslant
j<i\leqslant{\bf F}_K}
\beta_{\sigma(i)}\beta_{\sigma(j)}\omega_0(n_{\sigma(j)},n_{\sigma(i)})},
$$
where
$$
M(\hat{K})=\left\{((\beta_i)_{i=1}^{{\bf F}_K},\sigma)\,\bigg|\,\begin{array}{ll}
&\beta_i\geqslant
0,\;\sum_{i=1}^{{\bf F}_K}\beta_i\hat{h}_i=1,\;\sum_{i=1}^{{\bf F}_K}\beta_iJ n_i\in V_0^{n,k},\\
&\sum_{1\leqslant
j<i\leqslant{\bf F}_K}
\beta_{\sigma(i)}\beta_{\sigma(j)}\omega_0(n_{\sigma(j)},n_{\sigma(i)})>0,\;\sigma\in S_{{\bf F}_K}
\end{array}
\right\}.
$$
Now we are in position to prove that $M(\hat{K})$ is equal to $M({K})$ in (\ref{e:1.6}).
 We only need to prove $M(\hat{K})\subset M({K})$
because of obvious reasons.
 Since $((\beta_i)_{i=1}^{{\bf F}_K},\sigma)
\in M(\hat{K})$ satisfies
$$
1=\sum_{i=1}^{{\bf F}_K}\beta_i\hat{h}_i=\sum_{i=1}^{{\bf F}_K}\beta_i{h}_i-\Big\langle p, \sum_{i=1}^{{\bf F}_K}\beta_in_i\Big\rangle,
$$
 it suffices to prove $\langle p, \sum_{i=1}^{{\bf F}_K}\beta_in_i\rangle=0$.
 Note that $\sum_{i=1}^{{\bf F}_K}\beta_i Jn_i\in V_0^{n,k}$.
 We have
 \begin{eqnarray*}
 \Big\langle p, \sum_{i=1}^{{\bf F}_K}\beta_in_i\Big\rangle&=&\omega_0\Big(p, \sum_{i=1}^{{\bf F}_K}\beta_i Jn_i\Big)=0
 \end{eqnarray*}
 because $\mathbb{R}^{n,k}$ and $V_0^{n,k}$ are $\omega_0$-orthogonal.
   Hence $M(\hat{K})\subset M({K})$.
 \end{proof}

\begin{proof}[Proof of Theorem~\ref{th:converse}]
Let $p\in D\cap L\cap\mathbb{R}^{1,0}$, define
$\phi:\mathbb{R}^2\rightarrow\mathbb{R}^2, x\mapsto x-p$.
As in \cite[\S3]{JinLu1917} we have $c_{\rm LR}(D, D\cap\mathbb{R}^{1,0})=c_{\rm LR}(\phi(D), \phi(D)\cap\mathbb{R}^{1,0})$ and
$$
c_{\rm LR}(D_1, D_1\cap\mathbb{R}^{1,0})=c_{\rm LR}(\phi(D_1), \phi(D_1)\cap\mathbb{R}^{1,0}),\quad
c_{\rm LR}(D_2, D_2\cap\mathbb{R}^{1,0})=c_{\rm LR}(\phi(D_2), \phi(D_2)\cap\mathbb{R}^{1,0}).
$$
Thus we can assume $0\in D\cap L\cap\mathbb{R}^{1,0}$ below.

Let $H^+:=\{(x,y)\in\mathbb{R}^2\,|\, y\geqslant0\}$, $H^-:=\{(x,y)\in\mathbb{R}^2\,|\, y\leqslant0\}$,
and write $K^+=H^+\cap K$ and $K^-=H^-\cap K$ for any subset $K\subset\mathbb{R}^2$.
On each of $\partial D$, $\partial D_1$ and $\partial D_2$ there only exist two
generalized leafwise chords for $\mathbb{R}^{1,0}$, that is,
$(\partial D)^+$ and $(\partial D)^-$ on $\partial D$,
$(\partial D_1)^+$ and $(\partial D_1)^-$ on $\partial D_1$,
$(\partial D_2)^+$ and $(\partial D_2)^-$ on $\partial D_2$.
Note that  a GLC $x$ on $\partial D$ for $\mathbb{R}^{1,0}$ and the line segment $D\cap\mathbb{R}^{1,0}$  form a loop $\gamma$
and that $\langle -J\dot{z},z\rangle$ vanishes along the line segment $D\cap \mathbb{R}^{1,0}$.
Using these and Stokes theorem we deduce that $A(x)=\int_x qdp=\int_\gamma qdp$
is equal to the symplectic area of the domain surrounded by $\gamma$. Hence
\begin{eqnarray*}
&&c_{\rm LR}(D, D\cap\mathbb{R}^{1,0})=\min\{{\rm Area}(D^+),{\rm Area}(D^-)\},\\
&&c_{\rm LR}(D_1, D_1\cap\mathbb{R}^{1,0})=\min\{{\rm Area}(D_1^+),{\rm Area}(D_1^-)\},\\
&&c_{\rm LR}(D_2, D_2\cap\mathbb{R}^{1,0})=\min\{{\rm Area}(D_2^+),{\rm Area}(D_2^-)\}.
\end{eqnarray*}
Assume without loss of generality that $c_{\rm LR}(D, D\cap\mathbb{R}^{1,0})={\rm Area}(D^+)$. Then
\begin{eqnarray*}
 c_{\rm LR}(D_1, D_1\cap\mathbb{R}^{1,0})+c_{\rm LR}(D_2, D_2\cap\mathbb{R}^{1,0})&\leqslant& {\rm Area}(D_1\cap D^+)+{\rm Area}(D_2\cap D^+)\\
 &=&{\rm Area}(D^+)=c_{\rm LR}(D, D\cap\mathbb{R}^{1,0}).
 \end{eqnarray*}
\end{proof}

\medskip
\begin{tabular}{l}
 School of Mathematical Sciences, Beijing Normal University\\
 Laboratory of Mathematics and Complex Systems, Ministry of Education\\
 Beijing 100875, The People's Republic of China\\
 E-mail address: shikun@mail.bnu.edu.cn,\hspace{5mm}gclu@bnu.edu.cn\\
\end{tabular}

\end{document}